\documentclass[11pt]{article}

\usepackage[margin=1.2in]{geometry}

\usepackage{amsmath,amsthm,amsfonts,amssymb,latexsym}
\usepackage{hyperref}
\usepackage{comment}
\usepackage{enumerate}
\usepackage{color}
\usepackage[dvipsnames]{xcolor}

\usepackage[shortlabels]{enumitem}
\usepackage{tikz}

\usepackage{xparse,etoolbox}

\usepackage{sectsty}
\allsectionsfont{\centering}

\usepackage{parskip}
\makeatletter
\def\thm@space@setup{%
  \thm@preskip=\parskip \thm@postskip=0pt
}
\makeatother

\begin{document}

\theoremstyle{plain}

\newtheorem{thm}{Theorem}[section]

\newtheorem{lem}[thm]{Lemma}
\newtheorem{Problem B}[thm]{Problem B}

\newtheorem{pro}[thm]{Proposition}
\newtheorem{conj}[thm]{Conjecture}
\newtheorem{cor}[thm]{Corollary}
\newtheorem{que}[thm]{Question}
\newtheorem{prob}[thm]{Problem}

\theoremstyle{definition}
\newtheorem{rem}[thm]{Remark}
\newtheorem{defin}[thm]{Definition}
\newtheorem{hyp}[thm]{Hypothesis}
\newtheorem{exam}[thm]{Example}

\theoremstyle{plain}
\newtheorem*{thmA}{Theorem A}
\newtheorem*{thmB}{Theorem B}
\newtheorem*{corB}{Corollary B}
\newtheorem*{thmC}{Theorem C}
\newtheorem*{thmD}{Theorem D}
\newtheorem*{thmE}{Theorem E}
 
\newtheorem*{thmAcl}{Main Theorem$^{*}$}
\newtheorem*{thmBcl}{Theorem B$^{*}$}
\newcommand{\dd}{\mathrm{d}}

\theoremstyle{plain}
\newtheorem{theoA}{Theorem}

\theoremstyle{plain}
\newtheorem{probA}[theoA]{Problem}

\theoremstyle{plain}
\newtheorem{condA}[theoA]{Condition}

\theoremstyle{plain}
\newtheorem{paraA}[theoA]{Parametrisation}

\theoremstyle{plain}
\newtheorem{corA}[theoA]{Corollary}

\renewcommand{\thetheoA}{\Alph{theoA}}

\renewcommand{\thecorA}{\Alph{corA}}

\newcommand{\wh}[1]{\widehat{#1}} 
\newcommand{\miquelcomment}{\textcolor{blue}}
\newcommand{\ncomment}{\textcolor{magenta}}

\newcommand{\Maxn}{\operatorname{Max_{\textbf{N}}}}
\newcommand{\Syl}{\operatorname{Syl}}
\newcommand{\Hall}{\operatorname{Hall}}
\newcommand{\Lin}{\operatorname{Lin}}
\newcommand{\U}{\mathbf{U}}
\newcommand{\nav}{\mathrm{Nav}}
\newcommand{\R}{\mathbf{R}}
\newcommand{\dl}{\operatorname{dl}}
\newcommand{\Con}{\operatorname{Con}}
\newcommand{\rdz}{\operatorname{rdz}}
\newcommand{\rdzo}{\operatorname{rdz}^{\circ}}
\newcommand{\cl}{\operatorname{cl}}
\newcommand{\Stab}{\operatorname{Stab}}
\newcommand{\Aut}{\operatorname{Aut}}
\newcommand{\Ker}{\operatorname{Ker}}
\newcommand{\InnDiag}{\operatorname{InnDiag}}
\newcommand{\fl}{\operatorname{fl}}
\newcommand{\irr}{\operatorname{Irr}}
\newcommand{\ibr}{\operatorname{IBr}}
\newcommand{\FF}{\mathbb{F}}
\newcommand{\CL}{\mathfrak{Cl}}
\newcommand{\EE}{\mathbb{E}}
\newcommand{\Alp}{\mathrm{Alp}}
\newcommand{\Alpr}{\mathrm{Alp_r}}
\newcommand{\normal}{\unhld}
\newcommand{\sn}{\normal\normal}
\newcommand{\Bl}{\mathrm{Bl}}
\newcommand{\NN}{\mathbb{N}}

\newcommand{\N}{\mathbf{N}}
\newcommand{\bfC}{\mathbf{C}}
\newcommand{\bfO}{\mathbf{O}}
\newcommand{\bfF}{\mathbf{F}}
\newcommand{\Irr}{\mathrm{Irr}}
\def\GGG{{\mathcal G}}
\def\HHH{{\mathcal H}}
\def\HH{{\mathcal H}}
\def\irra#1#2{{\rm Irr}_{#1}(#2)}

\renewcommand{\labelenumi}{\upshape (\roman{enumi})}

\newcommand{\PSL}{\operatorname{PSL}}
\newcommand{\PSU}{\operatorname{PSU}}
\newcommand{\alt}{\operatorname{Alt}}

\providecommand{\V}{\mathrm{V}}
\providecommand{\E}{\mathrm{E}}
\providecommand{\ir}{\mathrm{Irm_{rv}}}
\providecommand{\Irrr}{\mathrm{Irm_{rv}}}
\providecommand{\re}{\mathrm{Re}}

\numberwithin{equation}{section}
\def\irrp#1{{\rm Irr}_{p'}(#1)}

\def\ibrrp#1{{\rm IBr}_{\Bbb R, p'}(#1)}
\def\C{{\mathbb C}}

\def\isoc{{\succeq_c}}

\def\isob{{\succeq_b}}

\newcommand{\wt}[1]{\widetilde{#1}} 

\def\Rad{{\rm Rad}}
\def\Rado{{\rm Rad}^{\circ}}

\def\o{{\bf O}}
\def\c{{\bf C}}
\def\n{{\bf N}}
\def\z{{\bf Z}}
\def\F{{\bf F}}
\def\I{{\bf I}}
\def\P{{\mathcal{P}}}
\def\Q{{\mathcal{Q}}}
\def\R{{\mathcal{R}}}
\def\W{{\mathcal{W}}}
\def\D{{\mathcal{D}}}
\def\Wr{{\mathcal{W}_{\rm r}}}
\def\supp{{\rm{Supp}}}

\def\irr#1{{\rm Irr}(#1)}
\def\irrp#1{{\rm Irr}_{p^\prime}(#1)}
\def\irrq#1{{\rm Irr}_{q^\prime}(#1)}
\def \aut#1{{\rm Aut}(#1)}
\def\cent#1#2{{\bf C}_{#1}(#2)}
\def\norm#1#2{{\bf N}_{#1}(#2)}
\def\zent#1{{\bf Z}(#1)}
\def\syl#1#2{{\rm Syl}_#1(#2)}
\def\normal{\unlhd}
\def\oh#1#2{{\bf O}_{#1}(#2)}
\def\Oh#1#2{{\bf O}^{#1}(#2)}
\def\det#1{{\rm det}(#1)}
\def\gal#1{{\rm Gal}(#1)}
\def\ker#1{{\rm ker}(#1)}
\def\normalm#1#2{{\bf N}_{#1}(#2)}
\def\alt#1{{\rm Alt}(#1)}
\def\iitem#1{\goodbreak\par\noindent{\bf #1}}
   \def \mod#1{\, {\rm mod} \, #1 \, }
\def\sbs{\subseteq}

\def\gc{{\bf GC}}
\def\ngc{{non-{\bf GC}}}
\def\ngcs{{non-{\bf GC}$^*$}}
\newcommand{\notd}{{\!\not{|}}}
\newcommand{\Bpi}{\mathbf{B}_\pi}
\newcommand{\Ipi}{\mathbf{I}_\pi}
\newcommand{\Z}{\mathbf{Z}}

\newcommand{\bG}{\mathbf{G}}
\newcommand{\bL}{\mathbf{L}}
\newcommand{\bH}{\mathbf{H}}
\newcommand{\bM}{\mathbf{M}}

\newcommand{\ty}[1]{\mathsf{#1}}

\newcommand{\cE}{\mathscr{E}}

\newcommand{\Out}{{\mathrm {Out}}}
\newcommand{\Mult}{{\mathrm {Mult}}}
\newcommand{\Inn}{{\mathrm {Inn}}}
\newcommand{\Fong}{{\mathrm{Fong}}}
\newcommand{\IBRL}{{\mathrm {IBr}}_{\ell}}
\newcommand{\IBRP}{{\mathrm {IBr}}_{p}}
\newcommand{\bl}{{\mathrm{bl}}}
\newcommand{\cd}{\mathrm{cd}}
\newcommand{\ord}{{\mathrm {ord}}}
\def\id{\mathop{\mathrm{ id}}\nolimits}
\renewcommand{\Im}{{\mathrm {Im}}}
\newcommand{\Ind}{{\mathrm {Ind}}}
\newcommand{\diag}{{\mathrm {diag}}}
\newcommand{\soc}{{\mathrm {soc}}}
\newcommand{\End}{{\mathrm {End}}}
\newcommand{\sol}{{\mathrm {sol}}}
\newcommand{\Hom}{{\mathrm {Hom}}}
\newcommand{\Mor}{{\mathrm {Mor}}}
\newcommand{\Mat}{{\mathrm {Mat}}}
\def\rank{\mathop{\mathrm{ rank}}\nolimits}
\newcommand{\Tr}{{\mathrm {Tr}}}
\newcommand{\tr}{{\mathrm {tr}}}
\newcommand{\Gal}{{\rm Gal}}
\newcommand{\Spec}{{\mathrm {Spec}}}
\newcommand{\ad}{{\mathrm {ad}}}
\newcommand{\Sym}{{\mathrm {Sym}}}
\newcommand{\Char}{{\mathrm {Char}}}
\newcommand{\pr}{{\mathrm {pr}}}
\newcommand{\rad}{{\mathrm {rad}}}
\newcommand{\abel}{{\mathrm {abel}}}
\newcommand{\PGL}{{\mathrm {PGL}}}
\newcommand{\PCSp}{{\mathrm {PCSp}}}
\newcommand{\PGU}{{\mathrm {PGU}}}
\newcommand{\codim}{{\mathrm {codim}}}
\newcommand{\ind}{{\mathrm {ind}}}
\newcommand{\Res}{{\mathrm {Res}}}
\newcommand{\Lie}{{\mathrm {Lie}}}
\newcommand{\Ext}{{\mathrm {Ext}}}
\newcommand{\EBr}{{\mathrm {EBr}}}
\newcommand{\Alt}{{\mathrm {Alt}}}
\newcommand{\AAA}{{\sf A}}
\newcommand{\SSS}{{\sf S}}
\newcommand{\DDD}{{\sf D}}
\newcommand{\QQQ}{{\sf Q}}
\newcommand{\CCC}{{\sf C}}
\newcommand{\SL}{{\mathrm {SL}}}
\newcommand{\Sp}{{\mathrm {Sp}}}
\newcommand{\PSp}{{\mathrm {PSp}}}
\newcommand{\SU}{{\mathrm {SU}}}
\newcommand{\GL}{{\mathrm {GL}}}
\newcommand{\GU}{{\mathrm {GU}}}
\newcommand{\Br}{{\mathrm{Br}}}
\newcommand{\Spin}{{\mathrm {Spin}}}
\newcommand{\CC}{{\mathbb C}}
\newcommand{\CB}{{\mathbf C}}
\newcommand{\RR}{{\mathbb R}}
\newcommand{\QQ}{{\mathbb Q}}
\newcommand{\ZZ}{{\mathbb Z}}
\newcommand{\bfN}{{\mathbf N}}
\newcommand{\bfZ}{{\mathbf Z}}
\newcommand{\PP}{{\mathbb P}}
\newcommand{\cG}{{\mathcal G}}
\newcommand{\cH}{{\mathcal H}}
\newcommand{\cQ}{{\mathcal Q}}
\newcommand{\GA}{{\mathfrak G}}
\newcommand{\cT}{{\mathcal T}}
\newcommand{\cL}{{\mathcal L}}
\newcommand{\IBr}{\mathrm{IBr}}
\newcommand{\cS}{{\mathcal S}}
\newcommand{\cR}{{\mathcal R}}
\newcommand{\GCD}{\GC^{*}}
\newcommand{\TCD}{\TC^{*}}
\newcommand{\FD}{F^{*}}
\newcommand{\GD}{G^{*}}
\newcommand{\HD}{H^{*}}
\newcommand{\GCF}{\GC^{F}}
\newcommand{\TCF}{\TC^{F}}
\newcommand{\PCF}{\PC^{F}}
\newcommand{\GCDF}{(\GC^{*})^{F^{*}}}
\newcommand{\RGTT}{R^{\GC}_{\TC}(\theta)}
\newcommand{\RGTA}{R^{\GC}_{\TC}(1)}
\newcommand{\Om}{\Omega}
\newcommand{\eps}{\epsilon}
\newcommand{\varep}{\varepsilon}
\newcommand{\dz}{\mathrm{dz}}
\newcommand{\dzo}{\mathrm{dz}^\circ}
\newcommand{\Co}{\mathcal{C}^\circ}
\newcommand{\al}{\alpha}

\newcommand{\chis}{\chi_{s}}
\newcommand{\sigmad}{\sigma^{*}}
\newcommand{\PA}{\boldsymbol{\alpha}}
\newcommand{\gam}{\gamma}
\newcommand{\lam}{\lambda}
\newcommand{\la}{\langle}
\newcommand{\genf}{F^*}
\newcommand{\ra}{\rangle}
\newcommand{\hs}{\hat{s}}
\newcommand{\htt}{\hat{t}}
\newcommand{\tG}{\hat G}
\newcommand{\St}{\mathsf {St}}
\newcommand{\bfs}{\boldsymbol{s}}
\newcommand{\bfl}{\boldsymbol{\lambda}}
\newcommand{\tn}{\hspace{0.5mm}^{t}\hspace*{-0.2mm}}
\newcommand{\ta}{\hspace{0.5mm}^{2}\hspace*{-0.2mm}}
\newcommand{\tb}{\hspace{0.5mm}^{3}\hspace*{-0.2mm}}
\def\skipa{\vspace{-1.5mm} & \vspace{-1.5mm} & \vspace{-1.5mm}\\}
\newcommand{\tw}[1]{{}^#1\!}
\renewcommand{\mod}{\bmod \,}

\let\ti=\times
\let\la=\lambda
\let\eps=\epsilon

\marginparsep-0.5cm

\newcommand{\blocktheorem}[1]{%
  \csletcs{old#1}{#1}
  \csletcs{endold#1}{end#1}
  \RenewDocumentEnvironment{#1}{o}
    {\par\addvspace{1.5ex}
     \noindent\begin{minipage}{\textwidth}
     \IfNoValueTF{##1}
       {\csuse{old#1}}
       {\csuse{old#1}[##1]}}
    {\csuse{endold#1}
     \end{minipage}
     \par\addvspace{1.5ex}}
}

\blocktheorem{theoA}
\blocktheorem{probA}

\makeatletter
\def\blfootnote{\gdef\@thefnmark{}\@footnotetext}
\makeatother

\title{{\bf{\huge The Isaacs--Navarro--Wolf conjecture}}}

\author{Damiano Rossi}
\date{}

\blfootnote{\emph{$2020$ Mathematical Subject Classification:} $20$C$15$, $20$D$10$
}

\maketitle

\begin{abstract}
The Isaacs--Navarro--Wolf conjecture states that if $G$ is a finite solvable group and $x$ is an element of $G$ such that $\chi(x)$ does not vanish for all irreducible characters $\chi$ of $G$, then $x$ must be contained in some nilpotent normal subgroup. In this paper we present a proof of the Isaacs--Navarro--Wolf conjecture that was discovered with the use of artificial intelligence systems. 
\end{abstract}

\section{Introduction}

In the recent past, the use of artificial intelligence in mathematical research has increasingly been exploited to generate proofs and settle open problems and conjectures. AI systems and LLMs, which until recently were merely able to solve relatively elementary mathematical questions, have become surprisingly proficient at research-level problems. Most remarkably, OpenAI has recently announced a collection of ten results resolving or making substantial progress on long-standing open problems in several areas of mathematics and theoretical computer science \cite{OpenAI}, including the construction of the first example of a non-sofic group. Similarly, Anthropic has produced an improved bound related to the Riemann zeta function \cite{Anthropic}. These are just a few examples from a rapidly growing collection of results obtained by deploying artificial intelligence to tackle open problems in mathematics (we also mention \cite{Nav-Ser26} as a recent example from the representation theory of finite groups). These recent developments have led to discussions about the possible (yet now seemingly inevitable) consequences and implications for mathematical research, its community, and how the role and aims of mathematicians may change as these new systems become increasingly capable of tackling open problems (see, for instance, \cite{Kon25}, \cite{Avi26}, \cite{CJOT26}, \cite{Tao26}, and their references).

The present work originated as an experiment aimed at testing the capabilities of emerging artificial intelligence systems in representation theory of finite groups. The Isaacs--Navarro--Wolf conjecture was selected as a test case immediately after the proof of Gluck's conjecture was uploaded on the arXiv by Baoyu Zhang (see \cite{ProofGluck}). The two problems are known to be closely related due to a well-known, although insufficient, proof mechanism related to the theory of large orbits of solvable linear groups. Besides the possible feasibility of the experiment, the choice was also dictated by the author's curiosity to potentially settle a problem encountered in his Master's thesis \cite{Ros-Mon}. We therefore initiated a systematic attempt to prove the conjecture by using a combination of ChatGPT to discuss existing literature and known proof strategies as well as the coding agent Codex to generate computational experiments that would help the discussion with the AI system. The experiment turned out to be successful: after obtaining a quite unnatural proof, we spent considerable time verifying, understanding, and ultimately simplifying and rewriting its arguments in a clear and mathematically rigorous form. We mention that, while the present manuscript was being prepared for circulation, Quanfu Yan and Jiping Zhang independently posted another proof of the Isaacs--Navarro--Wolf conjecture \cite{ProofINW}. Thus, apparently independently, the same conjecture was selected as a target for an AI-assisted proof attempt. This is not surprising given the above discussion on the relation of the Isaacs--Navarro--Wolf conjecture with Gluck's conjecture.

We now move on to the actual mathematical content of the paper. Let $G$ be a finite group and denote by $\Irr(G)$ the set of its irreducible complex characters. An element $g\in G$ is said to be \emph{non-vanishing} if $\chi(g)\neq 0$ for every $\chi\in\Irr(G)$. In \cite{INW}, Isaacs, Navarro and Wolf introduced the study of non-vanishing elements and provided evidence that every non-vanishing element of a finite solvable group belongs to a normal nilpotent subgroup. In particular they proved this statement for all elements of odd order and, more generally, for solvable groups with abelian Sylow $2$-subgroups. This and related problems have subsequently been studied both for solvable and for arbitrary finite groups by several authors (see, for instance, \cite{Mor-Wol04}, \cite{BDS}, \cite{DPSS}, \cite{DPSSII}, \cite{DNPST}, \cite{Wolf14}, \cite{Wolf20}).

The study of a minimal counterexample to the Isaacs--Navarro--Wolf conjecture naturally leads to the study of related questions about solvable linear groups. This was already explained in their seminal paper. The structure of a minimal counterexample was further reduced by work of Wolf (see \cite{Wolf14} and \cite{Wolf20}) where the quasi-primitive case was settled. The proof presented here ultimately builds on these reductions and handles the imprimitive case using a series of reductions (see Theorem \ref{thm:Reduction to primitive local action} and Theorem \ref{thm:Coherence Theorem}) and by ultimately verifying a condition (see Definition \ref{def:Coherence}) for a finite list of possible cases arising in the minimal counterexample setting (see Theorem \ref{thm:Verification of coherence}). This ultimately leads to a proof of the following result.

\begin{theoA}
\label{thm:Main}
Let $G$ be a finite solvable group and $g$ an element of $G$ such that $\chi(g)$ does not vanish for every irreducible character $\chi$ of $G$. Then $g$ is contained in a normal nilpotent subgroup of $G$.
\end{theoA}

The paper is structured as follows: in Section \ref{sec:Reduction} we present a reduction theorem for imprimitive linear solvable groups to the case in which the block normalizer is primitive (this is a fairly standard reduction). In Section \ref{sec:Coherence} we introduce a condition that we call \textit{coherence} and is then used to produce the vanishing of a certain irreducible character (this is proved in Theorem \ref{thm:Coherence Theorem}). Then, in Section \ref{sec:Verification}, we show that the coherence condition from Definition \ref{def:Coherence} is satisfied under relevant assumptions that can be deduced in the presence of a minimal counterexample to the conjecture. In Section \ref{sec:Proof} we then obtain a final proof of Theorem \ref{thm:Main}.

\section{Reduction to primitive local actions}
\label{sec:Reduction}

A minimal counterexample to the Isaacs--Navarro--Wolf conjecture is known to lead to a setting in which a finite solvable group $H$ acts linearly and irreducibly on a finite vector space $V$. We will see this more precisely in the final proof in Section \ref{sec:Proof}. In this section we prove a reduction theorem for the case in which the action of $H$ on $V$ is imprimitive. First, we need an easy lemma.

\begin{lem}
\label{lem:Regular actions from abelian}
Let $H$ be a finite group acting primitively and faithfully on the set $\Omega$. If $1\neq P\unlhd H$ is a $p$-subgroup, then $P$ acts regularly on $\Omega$.
\end{lem}

\begin{proof}
First, let $1\neq Z:=\z(P)$ and observe that $Z$ is transitive on $\Omega$. Indeed $Z$ is normal in $H$ and hence the set of $Z$-orbits on $\Omega$ is an $H$-stable partition of $\Omega$. Since $H$ is primitive, we deduce that either $Z$ is transitive on $\Omega$, or every $Z$-orbit is a singleton. But the latter condition implies that $Z$ is contained in the center of the action of $H$ on $\Omega$, contradicting the faithfulness of the action. Next assume there exists some $\omega\in\Omega$ and $1\neq z\in\c_{Z}(\omega)$. Then, for every $x\in Z$, we deduce that 
\[(\omega^x)^z=(\omega^z)^x=\omega^x\]
and, since $Z$ is transitive on $\Omega$, we conclude that $Z$ acts trivially on $\Omega$, again contradicting the hypothesis. Finally, observe that $Z$ is self-centralizing in $\Sym(\Omega)$ thanks to \cite[Theorem 4.2A (v)]{Dix-Mor96}. Since $P$ centralizes $Z$, we deduce that $P=Z$ and hence that $P$ is regular on $\Omega$.
\end{proof}

We can now proceed to prove our reduction theorem for the imprimitive case. The assumptions of the following statement will be satisfied by a minimal counterexample to Theorem \ref{thm:Main} (see the final proof in Section \ref{sec:Proof}).

\begin{thm}
\label{thm:Reduction to primitive local action}
Let $H$ be a finite group acting imprimitively on a finite vector space $V$. Consider an element $q$ of $H$ such that the subgroup $E$ generated by the $H$-conjugates of $q$ is a $p$-subgroup of $H$ and such that
\begin{equation}
\label{eq:Fix elt in every orbit}
V=\bigcup\limits_{h\in H}\c_V(q^h).
\end{equation} 
Then there exists a non-trivial system of imprimitivity $\Omega$ such that
\begin{enumerate}
\item $E$ is contained in $K_\Omega:=\bigcap_{W\in\Omega}\n_H(W)$
\item $\n_H(W)$ acts primitively on $W$, for every $W\in\Omega$.
\end{enumerate}
\end{thm}

\begin{proof}
Denote by $\mathcal{S}$ the set of all non-trivial systems of imprimitivity $\Sigma$, for the action of $H$ on $V$, such that $E\leq K_\Sigma:=\bigcap_{U\in \Sigma}\n_H(U)$. First, we claim that $\mathcal{S}$ is not empty. For this, take a maximal non-trivial system of imprimitivity $\Sigma$ and observe that $H$ is primitive on $\Sigma$. If $E$ is not contained in $K_\Sigma$, then we deduce that $EK_\Sigma/K_\Sigma$ is a non-trivial $p$-subgroup of the group $H/K_\Sigma$ acting faithfully and primitively on $\Sigma$. By Lemma \ref{lem:Regular actions from abelian} we know that $EK_\Sigma/K_\Sigma$ is regular on $\Sigma$. It follows that $\n_E(U)\leq K_\Sigma$ for every block $U$ of $\Sigma$ and therefore, for every $0\neq u\in U$, we must have $\c_E(u)\leq K_\Sigma$. Now, by \eqref{eq:Fix elt in every orbit} there exists an $H$-conjugate $y$ of $q$ such that $y$ centralizes $u$. In particular, $y$ is contained in $\c_E(u)\leq K_{\Sigma}$. Since $K_{\Sigma}$ is normal in $H$, then all $H$-conjugates of $y$ are contained in $K_{\Sigma}$. Therefore $E\leq K_{\Sigma}$, a contradiction. 

Next, observe that $\mathcal{S}$ is finite, because $V$ is finite, and that it admits a poset structure induced by refinement of systems of imprimitivity. In particular there exists a system of imprimitivity $\Omega$ that is minimal in $\mathcal{S}$, that is, a system of imprimitivity $\Omega\in\mathcal{S}$ that admits no refinements in $\mathcal{S}$. In the remaining part of this proof we will show that $\Omega$ satisfies the conditions required in the above statement. First observe that (i) is satisfied by the definition of $\mathcal{S}$. Moreover, if $W\in \Omega$ and $W'$ is a non-trivial $\n_H(W)$-invariant subspace of $W$, then
\[\bigoplus\limits_{h\in\mathcal{T}}(W')^h\]
where $h\in H$ runs over a set of representatives $\mathcal{T}$ for the $\n_H(W)$-cosets in $H$, is a non-trivial $H$-invariant subspace of $V$. But $H$ is irreducible on $V$, hence we obtain a contradiction and the action of $\n_H(W)$ on $W$ must be irreducible.

It remains to show that $\n_H(W)$ is primitive on $W$ for every $W\in \Omega$. Assume this is not the case, fix $W\in \Omega$, and let $\Sigma_W$ be a non-trivial system of imprimitivity for the action of $\n_H(W)$ on $W$. Assume furthermore that $\Sigma_W$ is chosen to be maximal so that the action of $\n_H(W)$ on $\Sigma_W$ is primitive. We can extend $\Sigma_W$ to a non-trivial system of imprimitivity for the action of $H$ on $V$ by setting
\[\Sigma_V:=\bigcup\limits_{h\in \mathcal{T}}(\Sigma_W)^h.\]
Observe now that $\Sigma_V$ is a refinement of $\Omega$ and that the minimality of $\Omega$ in $\mathcal{S}$ implies that $\Sigma_V$ does not belong to $\mathcal{S}$. In other words, we have that $E$ is not contained in $K_{\Sigma_V}$. Next, fix $0\neq u\in U\in \Sigma_W$ and set
\[v:=\sum\limits_{h\in \mathcal{T}}u^h\neq 0.\]
By \eqref{eq:Fix elt in every orbit}, there exists an $H$-conjugate $y$ of $q$ that centralizes $v$. On the other hand, since $y\in E\leq K_\Omega$ we deduce that $(u^h)^y=u^h$, for every $h\in \mathcal{T}$. It follows that
\begin{equation}
\label{eq:Reduction to primitive local action, I}
(U^h)^y=U^h, \text{ for every }h\in\mathcal{T}.
\end{equation}
Remember that $E$ is not contained in $K_{\Sigma_V}$, hence $y$ cannot fix all blocks of $(\Sigma_W)^h$ for every $h\in\mathcal{T}$. Therefore, there exists $g\in\mathcal{T}$ such that $y$ does not belong to $(K_{\Sigma_W})^g$, where $K_{\Sigma_W}:=\bigcap_{U\in\Sigma_W}\n_{\n_H(W)}(U)$. Applying \eqref{eq:Fix elt in every orbit} with $h=g$, we get 
\begin{equation}
\label{eq:Reduction to primitive local action, II}
U^z=U
\end{equation}
for $z:=y^{g^{-1}}\in E$. Therefore, we have found an element $z\in E$ such that $z$ fixes the block $U$ of $\Sigma_W$ and $z$ does not belong to $K_{\Sigma_W}$. On the other hand, by Lemma \ref{lem:Regular actions from abelian}, we know that $1\neq EK_{\Sigma_W}/K_{\Sigma_W}$ is regular on $\Sigma_W$, contradicting \eqref{eq:Reduction to primitive local action, II}. This shows that $\n_H(W)$ is primitive on $W$ and the proof is now complete.
\end{proof}

Next, let $H$ be a finite group acting linearly on a vector space $V$. For a subspace $W$ of $V$ we denote by $\pi_W:\n_H(W)\to \n_H(W)/\c_H(W)$ the canonical projection. Then, for any subset $X$ of $H$, we set $X_W:=\pi_W(X\cap \n_H(W))$. With this notation, we have the following properties of a system of imprimitivity satisfying the conditions of Theorem \ref{thm:Reduction to primitive local action}.

\begin{lem}
\label{lem:Properties of primitive local action}
Let $H$ be a finite group acting imprimitively on a finite vector space $V$. Consider an element $q$ of $H$ satisfying \eqref{eq:Fix elt in every orbit} and let $\Omega$ be a non-trivial system of imprimitivity for the action of $H$ on $V$ such that $q$, and hence all its $H$-conjugates, belongs to $K_\Omega$. Then, for every $W\in \Omega$, we have:
\begin{enumerate}
\item For every $H$-conjugate $y$ of $q$, we have
\begin{equation}
\label{eq:Properties of primitive local action, I}
W=\bigcup\limits_{a\in H_W}\c_W(\pi_W(y)^a)
\end{equation}
\item If $\n_H(W)$ acts primitively on $W$, then $H_W$ acts primitively and faithfully on $W$.
\end{enumerate}
\end{lem}

\begin{proof}
Let $0\neq w\in W$. Fix a representative set $\mathcal{T}$ for the $\n_H(W)$-cosets in $H$, including the identity element, and consider the set $\mathcal{A}$ of elements of the form
\[\sum\limits_{h\in \mathcal{T}}u_h\]
where, for every $h\in\mathcal T$, the element $u_h$ is of the form
$(w^g)^h$ for some $g\in\n_H(W)$. Since $\mathcal{A}$ is $H$-invariant it must be a union of $H$-orbits on $V$. On the other hand by \eqref{eq:Fix elt in every orbit} it follows that $y$ fixes an element in each $H$-orbit on $V$. Therefore, there is some element $v=\sum_{h\in\mathcal{T}}u_h$ in $\mathcal{A}$ that is fixed by $y$. Since $y\in K_\Omega$, we deduce that $y$ fixes $u_h$ for every $h\in\mathcal{T}$. In particular for $h=1\in\mathcal{T}$, we deduce that $y$ fixes the $\n_H(W)$-conjugate $u$ of $w$. This shows that some $\n_H(W)$-conjugate of $y$ fixes $w$. In other words, we have proved that $w$ belongs to $C_W(\pi_W(y)^a)$ for some $a\in H_W$. Thus \eqref{eq:Properties of primitive local action, I} holds. For the remaining property it suffices to notice that the action of $H_W=\n_H(W)/\c_H(W)$ is faithful on $W$, while the other conditions required in (ii) are included in the assumption.
\end{proof}

\section{Coherence and non-vanishing elements}
\label{sec:Coherence}

In this section, we identify a mechanism to produce zeroes of characters. This will be used to obtain a contradiction in the final proof. Recall that for a finite group $H$ acting linearly on a vector space $V$, and for a subspace $W$ of $V$, we denote by $\pi_W:\n_H(W)\to \n_H(W)/\c_H(W)$ the canonical projection. We denote by $X_W:=\pi_W(X\cap \n_H(W))$ for any subset $X$ of $H$. For an element $y\in \n_H(W)$, we simply write $\pi_W(y)$ instead of $\pi_W(\{y\})$.

\begin{defin}
\label{def:Support}
With the above notation, suppose furthermore that $\Omega$ is a system of imprimitivity of $V$. Recall that $K_\Omega$ is the kernel of the action of $H$ on $\Omega$, that is, the intersection of all $\n_H(W)$ for $W\in \Omega$. For every $y\in K_\Omega$, we define the \textit{support} of $y$ to be the set
\[\supp(y):=\left\lbrace W\in\Omega\,\middle|\, \pi_W(y)\neq 1\right\rbrace.\]
Now, fix an element $q$ of $H$ contained in $K_\Omega$ and observe that, since $K_\Omega$ is normal in $H$, every $H$-conjugate of $q$ is also contained in $K_\Omega$. Then, for every $W\in\Omega$, we define the \textit{support} of $W$ relative to the $H$-class of $q$ to be the set
\[\supp(W):=\left\lbrace \pi_W(y)\,\middle|\,\pi_W(y)\neq 1,\, y \text{ is $H$-conjugate to }q \right\rbrace.\]
Finally, if $E$ denotes the subgroup of $H$ generated by all the $H$-conjugates $y$ of $q$, then we denote by $E'_W$ the derived subgroup of $E_W=E/\c_E(W)$ and set
\[\supp(W)':=\left\lbrace E_W'a\,\middle|\,a\in\supp(W)\right\rbrace.\]
\end{defin}

With the above notation, we now introduce the notion of a coherent quadruple. This notion will subsequently be shown to produce zeroes for some characters of $H$.


\begin{defin}
\label{def:Coherence}
Let $H$ be a finite group acting linearly on the finite vector space $V$. Let $q$ be an element of $H$ and denote by $E$ the normal subgroup of $H$ generated by the $H$-conjugates of $q$. Let $\Omega$ be a system of imprimitivity with $E\leq K_\Omega$. We say that the quadruple $(H,V,q,\Omega)$ is \textit{coherent} if the following conditions are satisfied:
\begin{enumerate}
\item[(C.1)]  For every $W\in \Omega$ with $\supp(W)$ not empty, the set $\supp(W)'$ has at least two elements and we have
\[\c_W(a)\cap \c_W(b)\neq 0\hspace{15pt} \Longrightarrow \hspace{15pt} E'_Wa=E'_Wb\]
for every $a,b\in\supp(W)$.
\item[(C.2)] There exists an $H$-invariant linear character $\delta$ of $E'$ such that, for all $H$-conjugates $y$ and $z$ of $q$, we have
\[\delta([y,z])\neq 1\]
whenever the cardinality of the set 
\[\mathcal{C}(y,z):=\left\lbrace W\in\Omega\,\middle|\, W\in\supp(y)\cap \supp(z), \, E'_Wy\neq E'_Wz\right\rbrace\]
is odd.
\end{enumerate}
\end{defin}

Next, we show how the coherence condition from Definition \ref{def:Coherence} can be exploited to obtain an important information on non-vanishing elements. This will be the crucial step in the final proof.

\begin{thm}
\label{thm:Coherence Theorem}
Let $H$ be a finite group acting imprimitively on the finite vector space $V$. Let $q$ be an element of $H$ satisfying \eqref{eq:Fix elt in every orbit} and denote by $E$ the subgroup of $H$ generated by the $H$-conjugates of $q$. Let $\Omega$ be a system of imprimitivity such that $E\leq K_\Omega$. Suppose furthermore that $(H,V,q,\Omega)$ is coherent. If $q$ is a non-vanishing element of $H$, then $q$ belongs to $\c_H(V)$.
\end{thm}

\begin{proof}
We assume that $q$ acts non-trivially on $V$ and use the conditions in Definition \ref{def:Coherence} to construct an irreducible character $\chi$ of $H$ such that $\chi(q)=0$. First, we observe that $\supp(y)$ is non-empty for every $H$-conjugate $y$ of $q$. For if $\supp(y)$ is empty, then $\pi_W(y)=1$ for every $W\in \Omega$ and therefore $y$ acts trivially on $V$. Since $C_H(V)$ is a normal subgroup of $H$, we deduce that $q$ acts trivially on $V$ against our assumption. Furthermore, we also have that $\supp(W)$ is non-empty for every $W\in \Omega$. In fact, take $U\in \supp(q)$ and notice that $\pi_U(q)\in \supp(U)$. Since the action of $H$ on $\Omega$ is transitive we can find $h\in H$ such that $U^h=W$ and hence $\pi_W(q^h)$ belongs to $\supp(W)$. In particular (C.1) of Definition \ref{def:Coherence} holds for every $W\in \Omega$. Furthermore, we know from Lemma \ref{lem:Properties of primitive local action} (i) that \eqref{eq:Properties of primitive local action, I} holds for every $W\in \Omega$ and every $H$-conjugate $y$ of $q$, that is, $\pi_W(y)$ fixes an element in every $H_W$-orbit on $W$.

Next, for every $W\in\Omega$ and every $a\in \supp(W)$, we fix an element $0\neq v_{W,a}\in \c_W(a)$. We claim that
\begin{equation}
\label{eq:Coherence Theorem, I}
v_{W,a}\in\c_W(b) \hspace{15pt}\Longrightarrow\hspace{15pt} E'_Wa=E'_Wb
\end{equation}
for all $b\in\supp(W)$. First, observe that $a$ fixes an element in every $H_W$-orbit on $W$ and therefore $\c_W(a)\neq 0$. Then, if $b\in\supp(W)$ and $v_{W,a}$ is centralized by $b$, then $\c_W(a)\cap \c_W(b)\neq 0$ and therefore $E'_Wa=E'_Wb$ by (C.1).

Next, we claim that
\begin{equation}
\label{eq:Coherence Theorem, II}
\delta([q,g])=1
\end{equation}
for every $g\in E$. Consider any irreducible character $\vartheta$ of $E$ lying above the linear character $\delta$ of $E'$ given by (C.2) of Definition \ref{def:Coherence}. If $\rho$ is a representation of $E$ affording the character $\vartheta$, then we have $\rho(x)=\delta(x)I$ for every $x\in E'$ since $\delta$ is $E$-invariant. Hence $\vartheta(gx)=\vartheta(g)\delta(x)$ for every $g\in E$ and $x\in E'$. In particular, using the equality $q^g=q[q,g]$, we deduce that
\begin{equation}
\label{eq:Coherence Theorem, III}
\vartheta(q)=\vartheta(q^g)=\vartheta(q)\delta([q,g])
\end{equation}
for every irreducible character $\vartheta$ of $E$ lying above $\delta$ and for every $g\in E$. Suppose now that \eqref{eq:Coherence Theorem, II} is false. Then \eqref{eq:Coherence Theorem, III} implies $\vartheta^h(q)=0$ for every irreducible character $\vartheta$ lying above $\delta$. In particular, since $\delta$ is $H$-invariant, we have $\vartheta^h(q)=0$ for every $h\in H$. If we now take any irreducible character $\chi$ of $H$ lying above $\vartheta$, then we obtain
\[\chi(q)=[\chi_E,\vartheta]\sum\limits_h\vartheta^h(q)=0\]
where $h$ runs over a representative set for the $H_\vartheta$-cosets in $H$. This contradicts our assumption that $q$ is non-vanishing and hence \eqref{eq:Coherence Theorem, II} must hold. Noticing that $[q^h,g]=[q,g^{h^{-1}}]^h$ and that $\delta$ is $H$-invariant, we deduce also that
\begin{equation}
\label{eq:Coherence Theorem, IV}
\delta([y,g])=1
\end{equation}
for every $H$-conjugate $y=q^h$ of $q$ and every $g\in E$. In particular, applying condition (C.2) of Definition \ref{def:Coherence}, we conclude that $|\mathcal{C}(y,z)|$ is even for all $H$-conjugates $y$ and $z$ of $q$.

Consider now an element $a_W$ of $\supp(W)$, for every $W\in \Omega$. We claim that there exists some $H$-conjugate $y$ of $q$ such that
\[E'_W\pi_W(y)=E'_Wa_W\]
for every $W\in \supp(y)$. To prove this fact, consider the elements $v_{W,a}$ fixed above and define the vector 
\[v:=\sum_{W\in \Omega}v_{W,a_W}.\]
It follows from \eqref{eq:Fix elt in every orbit} that there is some $H$-conjugate $y$ of $q$ such that $v^y=v$ and, since $y\in K_\Omega$, we obtain $v_{W,a_W}^y=v_{W,a_W}$ for all $W\in\Omega$. This, together with \eqref{eq:Coherence Theorem, I}, implies that $E'_Wa_W=E'_W\pi_W(y)$ for every $W\in \supp(y)$, as claimed. We now define the set
\[\mathcal{S}:=\left\lbrace (E'_Wa_W)_{W\in\Omega} \,\middle|\,a_W\in\supp(W)\right\rbrace\]
and, for every $H$-conjugate $y$ of $q$, the subset
\[\mathcal{S}_y:=\left\lbrace (E'_Wa_W)_{W\in\Omega}\in\mathcal{S} \,\middle|\,E'_Wa_W=E'_W\pi_W(y), \text{ for every }W\in\supp(y)\right\rbrace\]
of those tuples of $E'_W$-cosets $(E'_Wa_W)_{W\in\Omega}$ that are determined by $y$ on the indices contained in $\supp(y)$. We have just proved that the family $(\mathcal{S}_y)_y$, for $y$ running over all $H$-conjugates of $q$, covers the whole set $\mathcal{S}$. We now fix once and for all a minimal cover: namely we consider a subset $\mathcal{Y}$ of $H$-conjugates of $q$ such that $\mathcal{S}$ is covered by the family $(\mathcal{S}_y)_{y\in\mathcal{Y}}$ and $\mathcal{Y}$ has minimal size among all such subsets. We now fix $y_0\in\mathcal{Y}$ and observe that, by the minimality of $\mathcal{Y}$, there exists some tuple
\begin{equation}
\label{eq:Coherence Theorem, V}
f=(E'_Wa^f_W)_{W\in\Omega}\in\mathcal{S}_{y_0}\setminus\bigcup\limits_{y_0\neq y\in\mathcal{Y}}\mathcal{S}_y.
\end{equation}
Fix $W_0\in\supp(y_0)$ and let $y'_0$ be an $H$-conjugate of $q$ such that
\[E'_{W_0}\pi_{W_0}(y_0)\neq E'_{W_0}\pi_{W_0}(y'_0).\]
This is possible because by (C.1) of Definition \ref{def:Coherence} we know that $\supp(W_0)'$ has at least two elements. We now define a new element $f':=(E'_Wa^{f'}_W)_{W\in\Omega}$ of $\mathcal{S}$ by setting $a_W^{f'}:=a_W^f$ for all $W\neq W_0$ and
\[a^{f'}_{W_0}:=\pi_{W_0}(y_0').\]
By the choice of $\mathcal{Y}$ we can find $z\in \mathcal{Y}$ such that the tuple $f'$ belongs to $\mathcal{S}_z$. Observe that from the definition of $f'$ we cannot have $f'\in \mathcal{S}_{y_0}$ and hence $z\neq y_0$. Furthermore, we must have $W_0\in \supp(z)$ since otherwise we would have $f\in \mathcal{S}_z$, contradicting \eqref{eq:Coherence Theorem, V}. To conclude, we claim that $\mathcal{C}(y_0,z)=\{W_0\}$. This will contradict the fact that $|\mathcal{C}(y_0,z)|$ is even which, in turn, will contradict the initial assumption that $q$ does not belong to $\c_H(V)$. First, observe that $W_0$ belongs to $\mathcal{C}(y_0,z)$ because $W_0$ is contained in $\supp(y_0)\cap \supp(z)$ and
\[E'_{W_0}\pi_{W_0}(y_0)\neq E'_{W_0}\pi_{W_0}(y_0')=E'_{W_0}\pi_{W_0}(z).\]
On the other hand, if $W_0\neq U\in \supp(y_0)\cap \supp(z)$, then
\[E'_U\pi_U(z)=E'_Ua_U^{f'}=E'_Ua_U^{f}=E'_U\pi_U(y_0)\]
by the definition of $f'$. This proves our claim and the proof is now complete.
\end{proof}

\section{Verification of coherence for the minimal primitive families}
\label{sec:Verification}

We now verify that the coherence conditions from Definition \ref{def:Coherence} are satisfied under some additional assumptions. These additional assumptions will lead us to a final verification for a small set of cases described by Wolf in \cite[Theorem 2.1]{Wolf14}.

\begin{thm}
\label{thm:Verification of coherence}
Let $H$ be a finite solvable group acting faithfully and imprimitively on the finite vector space $V$. Consider an involution $q$ of $\F(H)$ satisfying \eqref{eq:Fix elt in every orbit} and a system of imprimitivity $\Omega$ such that $q\in K_\Omega$ and $\n_H(W)$ acts primitively on $W$, for every $W\in\Omega$. Denote by $E$ the $2$-subgroup of $H$ generated by the $H$-conjugates of $q$. Then, every $W\in \Omega$ has non-empty support $\supp(W)$ and:
\begin{enumerate}
\item The set $\supp(W)'$ has at least two elements;
\item For every $a,b\in\supp(W)$, we have 
\[\c_W(a)\cap\c_W(b)\neq 0 \hspace{15pt}\Longleftrightarrow\hspace{15pt} a=b;\]
\item There is a linear character $\epsilon_W$ of $E'_W$ such that
\[\epsilon_W([a,b])=\begin{cases}
1,\, & E'_Wa=E'_Wb
\\
-1,\, & E'_Wa\neq E'_Wb
\end{cases}\]
for all $a,b\in\supp(W)$.
\item For every $h\in H$ we have
\begin{equation}
\label{eq:Verification of coherence, I}
\epsilon_{W^h}(x^h)=\epsilon_W(x)
\end{equation}
for every $x\in E'_W$ and where $\epsilon_{W^h}$ is the linear character of $E'_{W^h}$ constructed in (iii) for the block $W^h$ of $\Omega$.
\end{enumerate}
\end{thm}

\begin{proof}
Since $q$ is an involution contained in $\F(H)$, the subgroup $E$ of $H$ must be contained in $\o_2(H)$ and it is therefore a $2$-subgroup of $H$ contained in $K_\Omega$. As in the previous section, for $W\in\Omega$, we let $\pi_W:\n_H(W)\to \n_H(W)/\c_H(W)$ be the canonical projection and define $X_W:=\pi_W(X\cap \n_H(W))$ for any subset $X$ of $H$. With this notation, observe that
\begin{equation}
\label{eq:Verification of coherence, II}
E_W=\langle\supp(W)\rangle
\end{equation}
for every $W\in\Omega$. Next, notice that for any $H$-conjugate $y$ of $q$ the support $\supp(y)$ is not empty. For if $y$ centralizes every block $W$ in $\Omega$, then $y$ centralizes $V$. However the action of $H$ on $V$ is faithful and therefore $y=1$, contradicting the assumption that $q$ is an involution. From this we deduce that the support $\supp(W)$ is not empty for every $W\in \Omega$. In fact, we know that there exists some $U\in\supp(q)$, and thus $\pi_U(q)\in \supp(U)$. On the other hand, since $H$ is transitive on $\Omega$, we can find $h\in H$ such that $U^h=W$ and hence $\pi_W(q^h)\in\supp(W)$. 

By Lemma \ref{lem:Properties of primitive local action} we know that $H_W$ acts faithfully and primitively on $W$, and that \eqref{eq:Properties of primitive local action, I} holds for every $\pi_W(y)\in\supp(W)$. Moreover, $E_W$ is a normal $2$-subgroup of $\F(H_W)$ generated by the involutions $a\in\supp(W)$. This shows that, for every $a\in\supp(W)$, the triple $(H_W,W,a)$ satisfies the hypotheses of \cite[Theorem 2.1]{Wolf14}. In particular $\c_{\F(H_W)}(w)$ has order $2$, for every $0\neq w\in W$. Hence, if $a,b\in \supp(W)$ and $0\neq w\in \c_W(a)\cap \c_W(b)$, then $a=b$. This proves (ii) from the statement. We will now prove properties (i) and (iii) by analyzing each of the three possibilities described in \cite[Theorem 2.1]{Wolf14}.

\underline{\textbf{Case 1.}} Let $r$ be a Mersenne prime, $W=\mathbb{F}_{r^2}$, and denote by $\Gamma(W)$ the semilinear group consisting of transformations $x\mapsto ax^\tau$ on $W$, for $a\in \mathbb{F}_{r^2}^\times$ and $\tau\in\Gal(\mathbb{F}_{r^2}/\mathbb{F}_r)\simeq C_2$. We fix a generator $\sigma$ of $\Gal(\mathbb{F}_{r^2}/\mathbb{F}_r)$, so that $x^\sigma=x^r$, for every $x\in W$. Then, $\Gamma(W)\simeq \mathbb{F}_{r^2}^\times\rtimes \langle\sigma\rangle\simeq C_{r^2-1}\rtimes C_2$. We have that $H_W=\F(H_W)=P\times C\subseteq \Gamma(W)$ where $P$ is a Sylow $2$-subgroup of $\Gamma(W)$ and $C$ is a cyclic group of odd order. Letting $\alpha\in \mathbb{F}_{r^2}^\times$ be any element of order $2(r+1)$, we deduce that $P$ is semidihedral of order $4(r+1)$ (remember that $r+1$ is a power of $2$ since $r$ is a Mersenne prime) and can be described as follows
\[P=\langle\alpha,\sigma\mid\alpha^{2(r+1)}=1=\sigma^2, \, \sigma\alpha\sigma=\alpha^r\rangle.\]
On the other hand observe that $C$ must be contained in $\z(\Gamma(W))=\mathbb{F}_r^\times$. In order to describe the support $\supp(W)$ and hence the group $E_W$, we need to compute the involutions of $H_W$. There is a unique central involution, namely $\alpha^{r+1}$, and $r+1$ non-central involutions of the form
\[\alpha_i:=\alpha^{2i}\sigma, \,\text{ for }i=0,\dots, r.\]
By Schur's lemma the central involution acts by scalar multiplication $-1$ on $W$, and hence $\c_W(\alpha^{r+1})=0$. Since all elements $a=\pi_W(y)\in \supp(W)$ satisfy \eqref{eq:Properties of primitive local action, I}, we conclude that $\supp(W)$ must be contained in the set of non-central involutions $\{\alpha_i\}_i$. On the other hand, since $\supp(W)$ is $H_W$-invariant and $\{\alpha_i\}_i$ consists of a single $H_W$-conjugacy class, we deduce that $\supp(W)=\{\alpha_i\}_i$. Therefore $E_W$ is the subgroup of $P$ generated by the non-central involutions, and hence
\[E_W=\langle \alpha^2,\sigma\rangle\simeq D_{2(r+1)}\]
the dihedral group of index $2$ in $P$. In particular, $E_W'=\langle\alpha^4\rangle$ is a cyclic group of order $(r+1)/2$ and we deduce that, for every $a=\alpha_i$ and $b=\alpha_j$ in $\supp(W)$, we have
\[E'_Wa=E'_Wb \hspace{15pt}\Longleftrightarrow\hspace{15pt} i\equiv j\pmod{2}.\]
In particular $|\supp(W)'|\geq 2$ and so (i) holds in this case. To prove (iii), let $\epsilon_W$ be the unique linear character of $E'_W$ of order $2$ with $\epsilon_W(\alpha^4)=-1$. Then, for every $a=\alpha_i$ and $b=\alpha_j$ in $\supp(W)$, we have
\[\epsilon_W([a,b])=(-1)^{i-j}=\begin{cases}
1,\, & E'_Wa=E'_Wb
\\
-1,\, & E'_Wa\neq E'_Wb
\end{cases}\]
as required by (iii).

\underline{\textbf{Case 2.}} In this case we have $|W|=5^2$, $H_W/\F(H_W)$ isomorphic to either $C_3$ or $S_3$, and $\F(H_W)=QZ$ the central product of the two normal subgroups $Q\simeq Q_8$ and $Z\simeq C_4$ of $H_W$ with $Q\cap Z=\Z(Q)$. We take generators $x,y$ of $Q$, so that
\[Q=\langle x,y\mid x^4=1, x^2=y^2, y^{-1}xy=x^{-1}\rangle\] 
and a generator $z$ of $Z$ and observe that $x^2=y^2=z^2$ generates $\Z(Q)$. As above, we compute the involutions of $QZ$. Since $Z=\Z(F(H_W))$, there is a unique central involution $z^2$ in $QZ$. The remaining non-central involutions of $QZ$ are 
\[xz,\, yz, \,xyz, \,xz^3, \,yz^3, \,xyz^3.\]
Arguing as in the previous case, observe that $z^2$ acts as the scalar $-1$ on $W$, so that $\c_W(z^2)=0$. From \eqref{eq:Properties of primitive local action, I}, we deduce that $\supp(W)$ must be contained in the set of non-central involutions. Furthermore, observe that the latter set is a single $H_W$-orbit. In fact, we know that $QZ$ has three classes of non-central involutions
\[\{xz,\, xz^3\}, \,\{yz,\, yz^3\},\,\{xyz,\, xyz^3\}\]
and that any element of order $3$ in $H_W$ permutes these three classes. As in the previous case, we deduce that $\supp(W)$ coincides with the set of non-central involutions. We then obtain $E_W=\langle\supp(W)\rangle=QZ$ and $E'_W=\langle z^2\rangle$. From this we immediately obtain $|\supp(W)'|=3\geq 2$ proving (i) for this case. Finally, letting $\epsilon_W$ be the unique linear character of order $2$ of $E'_W$ and observing that $\epsilon_W(z^2)=-1$, we directly compute that for every $a,b\in \supp(W)$ we have the equality required in (iii).

\underline{\textbf{Case 3.}} In this final case, we have $|W|=3^4$, $H_W/\F(H_W)$ isomorphic to $C_5$, $D_{10}$, or $F_{20}$ (observe that the groups $A_5$ and $S_5$ are excluded since $H$ is solvable in our setting), and $\F(H_W)$ the central product of $Q\simeq Q_8$ and $D\simeq D_8$. We can take generators $x,y$ of $Q$ and $f,g$ of $D$ so that
\[Q=\langle x,y\mid x^4=1,\, x^2=y^2,\,y^{-1}xy=x^{-1} \rangle\]
and
\[D=\langle f,g\mid f^4=1,\, g^2=1,\,gfg=f^{-1} \rangle\]
with $Q\cap D=\langle x^2\rangle=\langle f^2\rangle$. Once again, we compute the involutions of $\F(H_W)$. There is a unique central involution $x^2$ which, as in the previous cases, acts as the scalar $-1$ on $W$, so that $\C_W(x^2)=0$. We then must have $\supp(W)$ contained in the set of non-central involutions by \eqref{eq:Properties of primitive local action, I}. There are $10$ non-central involutions:
\[xf,\, yf, \,xyf, \,g, \,fg, \,xfx^2, \,yfx^2, \,xyfx^2, \,gx^2, \,fgx^2.\]
As before we observe that all non-central involutions of $\F(H_W)$ are $H_W$-conjugate: first we observe that there are five classes of involutions in $\F(H_W)$
\[\{xf,\, xfx^2\}, \,\{yf,\, yfx^2\},\,\{xyf,\, xyfx^2\}, \,\{g,\, gx^2\},\,\{fg,\, fgx^2\}\]
and then notice that any element of order $5$ in $H_W$ will transitively permute these five classes. In particular, we deduce that $\supp(W)$ coincides with the set of non-central involutions of $\F(H_W)$ and that therefore $E_W=QD$ and $E'_W=\langle x^2\rangle$. This implies that $|\supp(W)'|=5\geq 2$ proving (i). Moreover, taking $\epsilon_W$ to be the non-trivial linear character of $E'_W$ and observing that $\epsilon_W(x^2)=-1$ we can directly verify the equality required in (iii).

We finally prove (iv). First observe that if $\supp(W)$ is not empty, then so is $\supp(W^h)$ for every $h\in H$. Therefore applying $(iii)$ to $W^h$ we get a linear character $\epsilon_{W^h}$. In all cases, the character $\epsilon_W$ is uniquely defined: it is the unique linear character of $E'_W$ of order $2$. Therefore, $(\epsilon_W)^h$ must be the unique linear character of order $2$ of $(E'_W)^h=E'_{W^h}$ and \eqref{eq:Verification of coherence, I} follows.
\end{proof}

In the following Corollary we finally verify the conditions from Definition \ref{def:Coherence}.

\begin{cor}
\label{cor:Verification of coherence}
With the hypotheses of Theorem \ref{thm:Verification of coherence} we have that the quadruple $(H,V,q,\Omega)$ is coherent.
\end{cor}

\begin{proof}
The condition (C.1) of Definition \ref{def:Coherence} follows immediately from (i) and (ii) of Theorem \ref{thm:Verification of coherence}. To prove (C.2) of Definition \ref{def:Coherence}, recall that $\supp(W)$ is not empty for every $W\in \Omega$. We let $\epsilon_W$ be the linear characters of $E'_W$ constructed in Theorem \ref{thm:Verification of coherence}. We can inflate $\epsilon_W$ to a character $\delta_W$ of $E'\C_E(W)$ with $C_E(W)$ in its kernel. Restricting this character to $E'$ we obtain a linear character $\delta_W$ of $E'$ containing $E'\cap \c_E(W)$ in its kernel. We can now define
\[\delta=\prod\limits_{W\in \Omega}\delta_W.\]
Property (iv) of Theorem \ref{thm:Verification of coherence} implies that $\delta$ is $H$-invariant. Finally, let $y$ and $z$ be $H$-conjugates of $q$ such that $|\mathcal{C}(y,z)|$ is odd. If $W$ does not belong to the intersection $\supp(y)\cap \supp(z)$, then either $y$ or $z$ centralizes $W$, hence $[y,z]$ belongs to $E'\cap C_E(W)$ and $\delta_W([y,z])=1$. We conclude that
\[\delta([y,z])=\prod\limits_{W}\delta_W([y,z])=(-1)^{|\mathcal{C}(y,z)|}=-1\]
where in the above product $W$ runs over those blocks in $\Omega$ such that $W\in\supp(y)\cap \supp(z)$. This proves the condition (C.2) of Definition \ref{def:Coherence} and the proof is now complete.
\end{proof}

\section{Proof of Theorem \ref{thm:Main}}
\label{sec:Proof}

We are finally able to prove Theorem \ref{thm:Main}.

\begin{proof}[Proof of Theorem \ref{thm:Main}]
We let $G$ be a counterexample of minimal possible order in which $x$ is a non-vanishing element not contained in the Fitting subgroup $\F(G)$. By the argument of \cite[Theorem 4.4]{INW}, taking $H:=G/\c_G(K/L)$ for a suitable chief factor $K/L$ of $G$ not centralized by $x$, $V:=\Irr(K/L)$, and $q=\c_G(K/L)x$, we deduce that $H$ is a finite solvable group acting faithfully and irreducibly on the finite vector space $V$ and that $q$ is a non-vanishing involution of $\F(H)$ satisfying \eqref{eq:Fix elt in every orbit}. In particular, the subgroup $E$ of $H$ generated by the $H$-conjugates of $q$ is a normal $2$-subgroup of $H$. By \cite[Theorem 7]{Wolf20}, we also know that the action of $H$ on $V$ cannot be primitive. We can then apply Theorem \ref{thm:Reduction to primitive local action} to obtain a system of imprimitivity $\Omega$ such that $E\leq K_\Omega$ and $\n_H(W)$ acts primitively on $W$, for every block $W$ of $\Omega$. Now by Corollary \ref{cor:Verification of coherence} we deduce that the quadruple $(H,V,q,\Omega)$ is coherent. We can finally apply Theorem \ref{thm:Coherence Theorem} to obtain that $q$ belongs to $\c_H(V)$. This contradicts the fact that $H$ is faithful on $V$, and the proof is complete.
\end{proof}

\newcommand{\etalchar}[1]{$^{#1}$}

\vspace{1cm}

(D. Rossi) {\sc{Department of Mathematics, Rutgers University, Hill Center - Busch Campus, 110 Frelinghuysen Road, Piscataway, NJ 08854-8019, USA}}

\textit{Email address:} \href{mailto:damiano.rossi@rutgers.edu}{damiano.rossi@rutgers.edu}

\end{document}